\documentclass[preprint,12pt]{elsarticle}
\usepackage{amsmath}

\usepackage{graphics}
 \usepackage{graphicx}
 \usepackage{epsfig}

\usepackage{amssymb}
 \usepackage{amsthm}

\begin{document}

\begin{frontmatter}



\title{ Fourier spectral approximation for the convective Cahn-Hilliard equation in 2D case}

 \author[label1]{Xiaopeng Zhao\corref{cor1}}
 \ead{zhaoxiaopeng@jiangnan.edu.cn}
 \author[label2]{ Fengnan  Liu}
 \address[label1]{School of Science, Jiangnan University, Wuxi 214122, China}
\address[label2]{College of Mathematics, Jilin University, Changchun 130012, China}

\begin{abstract}
 In this paper, we consider
the Fourier spectral method for numerically solving the 2D convective
Cahn-Hilliard equation. The semi-discrete and fully discrete schemes
are established. Moreover, the existence, uniqueness and the optimal
error bound are also considered.
\end{abstract}


\end{frontmatter}

\newtheorem{theorem}{Theorem}
\newtheorem{lemma}[theorem]{Lemma}
\newtheorem{remark}[theorem]{Remark}
\newtheorem{definition}[theorem]{Definition}
\newproof{pf}{Proof}

\section{Introduction}

Suppose that $\Omega =[0, L_1]\times [0, L_2]$, $L_1, L_2>0$. We consider the following problem for the 2D convective Cahn-Hilliard equation. We seek a real-valued function $u(x,y,t)$ defined on $\Omega\times[0,T]$
\begin{equation}
\label{1-1} \frac{\partial u}{\partial t}+\gamma \Delta^2 u=\Delta
\varphi(u)+\nabla\cdot\psi(u),\quad (x,y)\in\Omega,~t\in(0,T],
\end{equation}
where
$\varphi(u)=\gamma_2u^3+\gamma_1 u^2-u,~\psi(u)=u^2$. $\gamma>0,$ $
\gamma_2>0$ and $\gamma_1$ are constants.
On the basis of physical considerations, as usual Eq.(\ref{1-1}) is
supplemented with the following boundary value conditions 
\begin{equation}
\label{1-2}u(x,y,t)=\Delta u(x,y,t)=0,\quad (x,y)\in\partial\Omega,
\end{equation}
and the
initial value condition 
\begin{equation}
\label{1-3} u(x,y,0)=u_0(x),~~(x,y)\in\Omega.
\end{equation}

Eq.(\ref{1-1}) is a typical fourth order parabolic equation, which
arises naturally as a continuous model for the formation of facets
and corners in crystal growth, see \cite{GAA,WSJ}. Here $u(x,t)$
denotes the slope of the interface. The convective term
$\nabla\cdot\psi(u)$ (see \cite{GAA}), stems from the effect of kinetic
that provides an independent flux of the order parameter, similar to
the effect of an external field in spinodal decomposition of a
driven system.

During the past years, many authors have paid much attention to the
convective Cahn-Hilliard equation. It was K. H. Kwek \cite{KKH} who
first studied the convective Cahn-Hilliard equation for the case
with convection, namely, $\psi(u)=u$. By some a priori estimates, he
proved the existence of a classical solution, and gave the error
estimates by the discontinuous Galerkin method. Zarksm et al
\cite{ZMPNG} investigate bifurcations of stations periodic solutions
of a convective Cahn-Hilliard equation, they described phase
separation in driven systems, and studied the stability of the main
family of these solutions.  Eden and Kalantarov \cite{Eden1, Eden2}
considered the convective Cahn-Hilliard equation with periodic
boundary conditions in one space dimension and three space
dimension. They established some result on the existence of a
compact attractor. Recently, Gao and Liu\cite{Gao} studied the instability of the traveling waves of the 1D convective Cahn-Hilliard equation. Zhao and Liu\cite{ZL1,ZL2} considered the optimal control problem for the convective Cahn-Hilliard equation in 1D and 2D case.
 For more
recent results on the convective Cahn-Hilliard equation, we refer
the reader to \cite{Liu2,Podolny, Zhao2} and the references
therein.

It is known to all, spectral methods are essentially discretization
methods for the approximate solution of partial differential
equations. They have the natural advantage in keeping the physical
properties of primitive problems \cite{Canuto,Chai, Ye1}.
On the other hand, until to now, there's no numerical results on the
convective Cahn-Hilliard equation by spectral methods. So, in this
paper, a Fourier spectral method for numerically solving problem
(\ref{1-1})-(\ref{1-3}) is developed.

\begin{remark}For the classical Cahn-Hilliard
equation(see\cite{Chai,Ye1,Elliott1, HYN}), there are two
important features: conservation of mass and the existence of
Lyapunov functional. These two properties play important roles both
in Cahn-Hilliard equation's mathematical theoretical analysis and
its numerical analysis. They are used to estimate the absolute
pointwise maximum value of the solution. However, for problem
(\ref{1-1})-(\ref{1-3}), the two important properties might not be existent.
This means that we should find another useful approach to estimate
the absolute pointwise maximum value of the solution.
\end{remark}

We now consider the Fourier spectral method for the problem
(\ref{1-1})-(\ref{1-3})., the existence of a solution locally in
time is proved by the standard Picard iteration, global classical
existence results can be found in \cite{Zhao2}. Adjusted to our
needs, the results is given in the following form:
\begin{theorem}\label{thm1.1}
Suppose that $u_0\in H^2(\Omega)\bigcap H_0^1(\Omega)$, Then problem
(\ref{1-1})-(\ref{1-3}) admits a unique solution $u$ such that
$$
u\in L^2([0,T];H^4(\Omega))\cap
L^\infty([0,T];H^2(\Omega)),\quad \forall T>0,
$$
\end{theorem}

This paper is organized as follows. In the next section, we consider
a semi-discrete Fourier spectral approximation, prove its existence
and uniqueness of the numerical solution and derive the error bound.
In section 3, we consider the full-discrete approximation for
problem (\ref{1-1})-(\ref{1-3}). Furthermore, we prove convergence
to the solution of the associated continuous problem. In the last
section, some numerical experiments which confirm our results are
performed.

Throughout this paper, we denote $L^2$, $L^p$, $L^{\infty}$, $H^k$
norm in $\Omega$ simply by $\|\cdot\|$, $\|\cdot\|_p$,
$\|\cdot\|_{\infty}$ and $\|\cdot\|_{H^k}$.

\section{Semi-discrete approximation}
In this section, we consider the semi-discrete approximation for
problem (\ref{1-1})-(\ref{1-3}).
 First of all, we recall some basic
results on the Fourier spectral method which will be used throughout
this paper. Let $L_1=L_2=2\pi$, $2N_1$, $2N_2$ be any positive integers. In the  continuation of this work, let $N_1=N_2=N$, $h=\frac{\pi}{N}$, $x_i=ih$, $y_j=jh$, $i,j\in \Lambda$, where $\Lambda=1,2,\cdots, 2N$. For any integer $2N>0$, we introduce the finite
dimensional subspace of $H^2(\Omega)\bigcap H_0^1(\Omega)$:
$$
S_N=\hbox{span}\left\{\sin k_1x\sin k_2y,\quad k_1,k_2\in\Lambda \right\},
$$
 Let $P_N:L^2(\Omega)\rightarrow S_N$
be an orthogonal projecting operator which satisfies:
\begin{eqnarray}
(u-P_Nu,v)=0,~~\forall v\in S_N. \label{2-1}
\end{eqnarray}
For operator $P_N$ and functions in $S_N$, we have the following
results (see \cite{Canuto, Ye1,HYN}).

(B1) $P_N$ commutes with derivation on $H^2$, i. e.
$$
P_N\Delta u=\Delta P_Nu,~~\forall u\in H^2(\Omega)\bigcap H_0^1(\Omega).
$$

(B2) For any real $0\leq\mu\leq\sigma$, there is a constant $c$,
such that
$$
\|u-P_Nu\|_{\mu}\leq cN^{\mu-\sigma}\|\nabla^{\sigma}u\|,~~\forall
u\in H^{\sigma}(\Omega).
$$

 We define the
Fourier spectral approximation: For each $N\geq1$, find
$$u_N(t)=\sum_{j=1}^Na_j(t)\sin k_1x\sin k_2y\in S_N$$ such that $\forall
v_N\in S_N$,
\begin{eqnarray}\label{2-2}
&&(\frac{\partial u_N}{\partial t},v_N)+\gamma(\Delta u_{N}, \Delta
v_{N})=(\varphi(u_N), \Delta v_{N})+(\nabla\cdot\psi(u_N), v_N),
\end{eqnarray}
for all $t\in[0,T]$ with $u_N(0)=P_Nu_0$.

Now, we are going to establish the existence, uniqueness et. al. of
the Fourier spectral approximation solution $u_N(t)$ for $t\geq 0$.
\begin{lemma}
Suppose that $u_0\in L^2(\Omega)$. Then, (\ref{2-2}) has a unique
solution $u_N(t)$ satisfying the following inequalities:
\begin{eqnarray}
\label{2-3}\|u_N(t)\|^2\leq e^{c_1t}\|u_0\|^2,~~\forall t\in(0,T),
\end{eqnarray}
and \begin{eqnarray} \label{2-3-1} \int_0^t\|\Delta
u_N(\tau)\|^2d\tau\leq\left(\frac{c_1t}{\gamma}e^{c_1t}+\frac1{\gamma}\right)\|u_0\|^2,~~t\in(0,T),
\end{eqnarray}
 where $c_1$ is a positive constant depends only on $\gamma$,
$\gamma_1$, $\gamma_2$ and the domain. \label{thm2.1}
\end{lemma}
\begin{proof}
Set $v_N=\sin j_1x\sin j_2y$ in (\ref{2-2}) for each $j~(1\leq j\leq N)$ to
obtain
\begin{eqnarray}
\label{2-4} \frac
d{dt}a_j(t)=f_j(a_1(t),a_2(t),\cdots,a_J(t)),~~j=1,2,\cdots,N,
\end{eqnarray}
where all $f_j:\mathbb{R}^N\rightarrow \mathbb{R}~(1\leq j\leq N)$
are smooth and locally Lipschitz continuous. Noticing that
$u_N(0)=P_Nu_0$, then
\begin{eqnarray}
\label{2-5} a_j(0)=(u_0,\phi_j),~~j=1,2,\cdots,N.
\end{eqnarray}
Using the theory of initial-value problems of the ordinary
differential equations, there is a time $T_N>0$ such that the
initial-value problem (\ref{2-4})-(\ref{2-5}) has a unique smooth
solution $(a_1(t),a_2(t),\cdots,a_N(t))$ for $t\in[0,T_N]$.

Setting $v_N=u_N$ in (\ref{2-2}), we obtain
$$
\frac12\frac d{dt}\|u_N\|^2+\gamma\|\Delta
u_N\|^2=(\varphi(u_N),\Delta u_N)+(\nabla\cdot\psi(u_N),u_N).
$$
Note that
$$
\varphi'(u_N)=3\gamma_2u_N^2+2\gamma_1u_N-u_N\geq-c_0=-\frac{\gamma_1^2}{3\gamma_2}-1.
$$
Thus
\begin{eqnarray}
(\varphi(u_N),\Delta u_N)&=&-(\varphi'(u_N)\nabla u_N,\nabla
u_N)\leq c_0\|\nabla u_N\|^2 \nonumber \\
&=&-c_0(u_N,\Delta u_N)\leq \frac{\gamma}2\|\Delta
u_N\|^2+\frac{c_0^2}{2\gamma}\|u_N\|^2.\nonumber
\end{eqnarray}
On the other hand, a simple calculation shows that
$$
(\nabla\cdot\psi(u_N),u_N)=\int_{\Omega}\nabla\cdot(u_N^2)u_Ndx=0,
$$
Summing up, we get
\begin{eqnarray}
\label{2-3-1-1} \frac d{dt}\|u_N\|^2+\gamma\|\Delta
u_N\|^2\leq\frac{c_0^2}{\gamma}\|u_N\|^2.
\end{eqnarray}
Using Gronwall's inequality, we obtain
$$
\|u_N\|^2\leq e^{\frac{c_0^2}{\gamma}t}\|u_N(0)\|^2\leq
e^{\frac{c_0^2}{\gamma}t}\|u_0\|^2,~~t\in(0,T).
$$
Setting $c_1=\frac{c_0^2}{\gamma}$, we get the conclusion
(\ref{2-3}). Integrating (\ref{2-3-1-1}) from $0$ to $t$, we get
\begin{equation}
\begin{aligned}
\int_0^t\|\Delta u_N(\tau)\|^2d\tau\leq&
\frac1{\gamma}\left(c_1\int_0^t\|u_N(\tau)\|^2d\tau+\|u_{N}(0)\|^2\right)\\\leq&
\left(\frac{c_1t}{\gamma}e^{c_1t}+\frac1{\gamma}\right)\|u_0\|^2.
\nonumber\end{aligned}\end{equation}
Hence, Lemma \ref{thm2.1} is proved.
\end{proof}

\begin{lemma}
Suppose that $u_0\in H^1_{0}(\Omega)$. Then, (\ref{2-2}) has a
unique solution $u_N(t)$ satisfying
\begin{eqnarray}
\label{2-6}\| \nabla u_N(t)\|^2\leq e^{c_2t}\|\nabla u_0\|^2+\tilde{c_2},~~t\in(0,T),
\end{eqnarray}
and
\begin{eqnarray}
\label{2-6-1} \int_0^t\|\nabla\Delta
u_N(\tau)\|^2d\tau\leq\frac{\tilde{c_3}t}{\gamma}+\frac1{\gamma}\|\nabla
u_0\|^2,~~t\in(0,T),
\end{eqnarray}
 where $c_2$ is a positive constant depends only on $\gamma$,
$\gamma_1$, $\gamma_2$ and the domain.\label{thm2.2}
\end{lemma}
\begin{proof} Setting $v_N=\Delta u_{N}$ in (\ref{2-2}), we obtain
\begin{equation}
\frac12\frac{d}{dt}\|\nabla u_N\|^2+\gamma\|\nabla\Delta
u_N\|^2=-\int_{\Omega}\Delta \varphi(u_N)\Delta
u_Ndx-\int_{\Omega}\nabla\cdot\psi(u_N)\Delta u_Ndx. \nonumber
\end{equation}
Note that
$$
\Delta \varphi(u_N)= (3\gamma_2u_N^2+2\gamma_1u_N-1)\Delta
 u_N+(6\gamma_2u_N+2\gamma_1)|\nabla u_N|^2.
$$
Hence 
\begin{eqnarray}
&&\frac12\frac{d}{dt}\|\nabla u_N\|^2+\gamma \|\nabla\Delta
u_N\|^2+\gamma_2\|u_N\Delta u_N\|^2 \nonumber
\\
&=&-\int_{\Omega}(2\gamma_2 u_N^2+2\gamma_1 u_N-1)|\Delta u_N|^2dx+\int_{\Omega}u_N^2\nabla\Delta u_Ndx\nonumber\\
&&-\int_{\Omega}2\gamma_1|\nabla u_N|^2\Delta
u_Ndx-6\int_{\Omega}\gamma_2u_N|\nabla u_N|^2\Delta u_Ndx \nonumber
\\
&\leq& \gamma_2\int_{\Omega}u_N^2|\Delta
u_N|^2dx+c\int_{\Omega}|\nabla u_N|^4dx+c\int_{\Omega}|\Delta
u_N|^2dx \nonumber
\\
&&+\frac{\gamma}{8}\int_{\Omega}|\nabla\Delta
u_N|^2dx+c\int_{\Omega}u_N^4dx. \nonumber
\end{eqnarray}
 On the other hand, by Nirenberg's inequality, we have
 \begin{eqnarray}
 \|\nabla u_N\|_4^4\leq \left(c'_1\|\nabla\Delta
 u_N\|^{\frac{1}{2}}\|u_N\|^{\frac{1}{2}}+c'_2\|u_N\|\right)^4\leq \frac{\gamma}{8}\|\nabla\Delta
u_N\|^2+c_3, \nonumber
\end{eqnarray}
\begin{eqnarray}
 \|u_N\|^4_4\leq \left(c'_1\|\nabla\Delta
 u_N\|^{\frac{1}{6}}\|u_N\|^{\frac56}+c'_2\|u_N\|\right)^4\leq \frac{\gamma}{8}\|\nabla\Delta
u_N\|^2+c_4,
 \nonumber
 \end{eqnarray}
 and
 \begin{eqnarray}
 \|\Delta u_N\|^2\leq\left(c'_1\|\nabla\Delta
 u_N\|^{\frac23}\|u_N\|^{\frac13}+c'_2\|u_N\|\right)^2\leq\frac{\gamma}{8}\|\nabla\Delta
u_N\|^2+c_5. \nonumber
\end{eqnarray}
Summing up, we immediately obtain
\begin{eqnarray}\label{2-6-1-1}
\frac d{dt}\|\nabla u_N\|^2+\gamma\|\nabla\Delta u_N\|^2\leq
2(c_3+c_4+c_5).
\end{eqnarray}
Using Nirenberg's inequality again, we get
\begin{equation}\label{ax}
\|\nabla u_N\|\leq c'_1\|\nabla\Delta u_N\|^{\frac13}\|u_N\|^{\frac23}+c'_2\|u_N\|.
\end{equation}
Adding (\ref{2-6-1-1}) and (\ref{ax}) together gives
\begin{eqnarray}\label{2-6-1-11}
\frac d{dt}\|\nabla u_N\|^2+c'_3\|\nabla u_N\|^2\leq
c_4'.
\end{eqnarray}
Therefore, Gronwall's inequality shows that
$$
\|\nabla u_N\|^2\leq e^{c'_3t}\|\nabla u_N(0)\|^2+\frac{c'_4}{c'_3}.
$$
Setting $c_2=c'_3$, $\tilde{c_2}=\frac{c'_4}{c'_3}$, we get the conclusion (\ref{2-6}).
Integrating (\ref{2-6-1-1}) from $0$ to $t$, we deduce that
$$
\int_0^t\|\nabla\Delta
u_N(\tau)\|^2d\tau\leq\frac2{\gamma}(c_3+c_4+c_5)t+\frac1{\gamma}\|\nabla
u_N(0)\|^2.
$$
Setting $\tilde{c_3}=2(c_3+c_4+c_5)$, we obtain (\ref{2-6-1}). Lemma \ref{thm2.2} is proved.
\end{proof}
\begin{remark}
\label{rem2.1} Based on Lemmas \ref{thm2.1}-\ref{thm2.2}, we obtain
the $H^1$-norm estimate of the numerical solution $u_N(t)$ for
problem (\ref{2-2}). Noticing that we consider the problem in 2D
case, by Sobolev's embedding theorem, we have
$H^1(\Omega)\rightharpoonup L^p(\Omega)$ for all $p<\infty$. Hence,
$$
\|u_N(t)\|_p\leq c\|u_N(t)\|_{H^1}\leq c_6,~~\forall p\in]1,\infty[,
$$
 where $c_6$ is a positive constant depends only on $\gamma$,
$\gamma_1$, $\gamma_2$ and the domain.
\end{remark}
\begin{lemma}
Suppose that $u_0\in H^2(\Omega)\bigcap H_0^1(\Omega)$. Then, (\ref{2-2}) has a
unique solution $u_N(t)$ satisfying
\begin{eqnarray}
\label{2-7}\| \Delta u_N(t)\|^2\leq e^{c_7t}\|\Delta u_0\|^2+c_8,~~t\in(0,T),
\end{eqnarray}
and
\begin{eqnarray}
\label{2-7-1} \int_0^t\|\Delta ^2u_N(\tau)\|^2d\tau\leq
\tilde{c}_7t+\tilde{c}_8\|\Delta u_0\|^2,~~t\in(0,T),
\end{eqnarray}
 where $c_7$, $c_8$ are positive constants depend only on $\gamma$,
$\gamma_1$, $\gamma_2$ and the domain.\label{thm2.3}
\end{lemma}
\begin{proof}
Setting $v_N=\Delta u_{N}$ in (\ref{2-2}), we obtain
\begin{equation}
\frac12\frac{d}{dt}\|\Delta
u_N\|^2+\gamma\|\Delta^2u_N\|^2=\int_{\Omega}\Delta^2u_N\Delta
\varphi(u_N)dx+\int_{\Omega}\Delta^2u_N\nabla \cdot\varphi(u_N)dx.
\nonumber
\end{equation}
By H\"{o}lder's inequality, we derive that 
\begin{equation}
\frac12\frac{d}{dt}\|\Delta u_N\|^2+\gamma\|\Delta^2u_N\|^2\leq
\frac1{\gamma}\|\Delta \varphi(u_N)\|^2+\frac1{\gamma}\|\nabla
\psi(u_N)\|^2+\frac{\gamma}{2}\|\Delta^2 u_N\|^2. \nonumber
\end{equation}
Noticing that 
\begin{eqnarray}
&&\|\Delta \varphi(u_N)\|^2 \nonumber\\&\leq&
2(\int_{\Omega}|\varphi'(u_N)|^2|\Delta
u_N|^2dx+\int_{\Omega}|\varphi{''}(u_N)|^2|\nabla u_N|^4dx)
\nonumber
 \\
 &\leq& 2[(\int_{\Omega}|\varphi'(u_N)|^3dx)^{\frac{2}{3}}(\int_{\Omega}|\Delta
 u_N|^6dx)^{\frac{1}{3}}+(\int_{\Omega}|\varphi{''}(u_N)|^6dx)^{\frac{1}{3}}(\int_{\Omega}|\nabla
 u_N|^6dx)^{\frac{2}{3}}]
 \nonumber
 \\
 &\leq &c_9[(\int_{\Omega}|\Delta
 u_N|^6dx)^{\frac{1}{3}}+(\int_{\Omega}|\nabla
 u_N|^6dx)^{\frac{2}{3}}],\nonumber
 \label{3-22}
\end{eqnarray}
 where $c_9$ is a positive constant depends only on $\gamma$,
$\gamma_1$, $\gamma_2$ and the domain. On the other hand, we have 
\begin{eqnarray}
\|\nabla \psi(u_N)\|^2=\int_{\Omega}u_N^2|\nabla u_N|^2dx
\leq\frac{1}{2}\int_{\Omega}u_N^4dx+\frac{1}{2}\int_{\Omega}|\nabla
u_N|^4dx \leq \frac{c_6^4}2+\frac12\|\nabla u_N\|^4_4. \nonumber
\end{eqnarray}
Using Nirenberg's inequality, we have
$$
\|\nabla u_N\|_4^4\leq \left(c'_1\|\Delta^2u_N\|^{\frac16}\|\nabla
u_N\|^{\frac56}+c'_2\|\nabla u_N\|\right)^4\leq
\varepsilon\|\Delta^2u_N\|^2+c_{\varepsilon},
$$
$$
\|\nabla u_N\|_6^4\leq \left(c'_1\|\Delta^2
u_N\|^{\frac{2}{9}}\|\nabla u_N\|^{\frac{7}{9}}+c'_2\|\nabla
u_N\|\right)^4\leq\varepsilon\|\Delta^2u_N\|^2+c_{\varepsilon},
$$and
$$
\|\Delta u_N\|_6^2\leq \left(c'_1\|\Delta^2
u_N\|^{\frac{5}{9}}\|\nabla u_N\|^{\frac{4}{9}}+'_2\|\nabla
u_N\|\right)\leq\varepsilon\|\Delta^2u_N\|^2+c_{\varepsilon}.
$$
Summing up, we derive that
\begin{eqnarray}
\label{2-7-1-1} \frac d{dt}\|\Delta
u_N\|^2+[\gamma-(\frac{4c_9}{\gamma}+\frac1{\gamma})\varepsilon]\|\Delta^2u_N\|^2\leq\frac{4c_9c_{\varepsilon}}{\gamma}+\frac{c_{\varepsilon}}{\gamma}+\frac{c_6^4}{\gamma},
\end{eqnarray}
where $\varepsilon$ is small enough, it satisfies $
\gamma-(\frac{4c_9}{\gamma}+\frac1{\gamma})\varepsilon>0$. By the Calderon-Zygmund type estimate, we get
$$
\frac d{dt}\|\Delta
u_N\|^2+\tilde{c_4}(\|\Delta u_N\|^2+\|\nabla\Delta u_N\|^2)\leq \tilde{c_5}.
$$
Therefore, Gronwall's inequality shows that
$$
\|\Delta u_N\|^2\leq e^{\tilde {c_4}t}\|\Delta u_N(0)\|^2+\frac{\tilde{c_5}}{\tilde{c_4}}.
$$
Setting
$c_7=\tilde{c_4}$, $c_8=\frac{\tilde{c_5}}{\tilde{c_4}}$,
we obtain (\ref{2-7}). Integrating (\ref{2-7-1-1}) form $0$ to $t$,
we obtain
$$
\int_0^t\|\Delta^2u_N(\tau)\|^2d\tau\leq\frac{4c_9c_{\varepsilon}+c_{\varepsilon}+c_6^4}{\gamma^2-(4c_9+1)\varepsilon}t+\frac{\gamma}{\gamma^2-(4c_9+1)\varepsilon}\|\Delta u_0\|^2,
$$
where $\tilde{c}_7=\frac{4c_9c_{\varepsilon}+c_{\varepsilon}+c_6^4}{\gamma^2-(4c_9+1)\varepsilon}$ and $\tilde{c}_8=\frac{\gamma}{\gamma^2-(4c_9+1)\varepsilon}$. Hence, we
get (\ref{2-7-1}). Lemma \ref{thm2.3} is proved.
\end{proof}
\begin{remark}
\label{rem2.2} Based on Lemmas \ref{thm2.1}-\ref{thm2.3}, we obtain
the $H^2$-norm estimate of the numerical solution $u_N(t)$ for
problem (\ref{2-2}). Noticing that we consider the problem in 2D
case, by Sobolev's embedding theorem, we have
$H^2(\Omega)\rightharpoonup L^{\infty}(\Omega)$, that is
$$
\|u_N(t)\|_{\infty}\leq c\|u_N(t)\|_{H^2}\leq c_{10},
$$
where $c_{10}$ is a positive constant depends only on $\gamma$,
$\gamma_1$, $\gamma_2$ and the domain.
\end{remark}

\begin{theorem}
\label{thm2.4} Suppose that $u_0\in H^2(\Omega)\bigcap H_0^1(\Omega)$. Then for any
$T>0$, problem (\ref{2-2}) admits a unique solution $u_N(x,t)$, such
that
$$u_N(x,t)\in L^{\infty}(0,T;H^2_{per}(\Omega))\bigcap
L^2(0,T;H^4_{per}(\Omega)).
$$
\end{theorem}
\begin{proof}
We are going to apply the Leray-Schauder fixed point theorem to complete the proof.

Define the linear space
$$
X=\left\{u_N\in L^{\infty}(0,T;H^2_{per}(\Omega))\bigcap
L^2(0,T;H^4_{per}(\Omega)); u|_{\partial\Omega}=0,u(x,y,0)=u_0\right\}.
$$
Clearly, $X$ is a Banach space. Define the associated operator $T$,
$$
T: X\rightarrow X,\quad u_N\rightarrow w_N,
$$
where $w$ is determined by the following linear problem:
\begin{equation}
\begin{aligned}&
\left(\frac{\partial w_N}{\partial t},v_N\right)+\gamma(\Delta w_N,\Delta v_N)=(\varphi(u_N),\Delta v_N)+(\nabla\cdot\psi(u_N),v_N),~\forall v_N\in S_N,
\\
&\frac{\partial w_N}{\partial n}|_{\partial\Omega}=\frac{\partial \Delta w_N}{\partial n}|_{\partial\Omega}=0,\quad w(x,y,0)=u_0.
\end{aligned}\nonumber\end{equation}
From the discussions in Lemmas \ref{thm2.1}-\ref{thm2.3} and by the contraction mapping principle, $T$ has a unique fixed point $u$, which is the desired solution of problem \ref{2-2}.

Because the proof of the uniqueness of the solution is easy, we omit it here.

Then, we complete the proof.

\end{proof}

Now, we estimate the error $\|u(t)-u_N(t)\|$. Denote
$\eta_N=u(t)-P_Nu(t)$ and $e_N=P_Nu(t)-u_N(t)$. From (\ref{1-1}) and
(\ref{2-2}), we get:
\begin{equation}\begin{aligned}
\label{2-23}&(e_{Nt},v_N)+\gamma( \Delta e_{N}, \Delta
v_{N}) \\
=&(\varphi(u)-\varphi(u_N),\Delta v_N)+(\nabla\cdot(\psi(u)-\psi(u_N)),v_N),~~\forall v_N\in S_N.
\end{aligned}\end{equation}
Set $v_N=e_N$ in (\ref{2-23}), we derive that
\begin{eqnarray}
\label{2-24}\frac12\frac d{dt}\|e_N\|^2+\gamma\| \Delta e_{N}\|^2
=(\varphi(u)-\varphi(u_N),\Delta e_N)-(\psi(u)-\psi(u_N),\nabla\cdot
e_N).\nonumber
\end{eqnarray}
By Theorem \ref{thm1.1}, we have
$\sup_{x\in\bar{\Omega}}|u(x,t)|\leq c_{11}$, where $c_{11}$ is a
positive constant depends only on $\gamma$, $\gamma_1$, $\gamma_2$
and the domain. Then
\begin{equation}\begin{aligned}&(\varphi(u)-\varphi(u_N),\Delta e_N)
\\
=&\gamma_2((u-u_N)(u^2+uu_N+u_N^2),\Delta e_N)
\\&+\gamma_1((u+u_N)(u-u_N),\Delta e_N)-(u-u_N,\Delta e_N)
\\
\leq&\gamma_2\sup_{x\in\bar{\Omega}}|u^2+uu_N+u_N^2|\cdot\|e_N+\eta_N\|\|\Delta
e_N\|
\\
&+\gamma_1\sup_{x\in\bar{\Omega}}|u+u_N|\cdot\|e_N+\eta_N\|\|\Delta
e_N\|+\|e_N+\eta_N\|\|\Delta e_N\|
\\
\leq&2\gamma_2(c_{10}^2+c_{10}c_{11}+c_{11}^2)(\|e_N\|\|\Delta
e_N\|+\|\eta_N\|\|\Delta e_N\|)
\\
&+2\gamma_1(c_{10}+c_{11})(\|e_N\|\|\Delta e_N\|+\|\eta_N\|\|\Delta
e_N\|)\\&+2(\|e_N\|\|\Delta e_N\|+\|\eta_N\|\|\Delta e_N\|)
\\
=&[2\gamma_2(c_{10}^2+c_{10}c_{11}+c_{11}^2)+2\gamma_1(c_{10}+c_{11})+2](\|e_N\|\|\Delta
e_N\|+\|\eta_N\|\|\Delta e_N\|)
\\
\leq&\frac{\gamma}4\|\Delta e_N\|^2+c_{12}(
\|e_N\|^2+\|\eta_N\|^2),\nonumber
\end{aligned}\end{equation}
where
$c_{12}=\frac{8}{\gamma}[\gamma_2(c_{10}^2+c_{10}c_{11}+c_{11}^2)+\gamma_1(c_{10}+c_{11})+1]^2$.
On the other hand, we have
\begin{eqnarray}
-(\psi(u)-\psi(u_N),\nabla\cdot e_N) &=&-((u-u_N)(u+u_N),\nabla\cdot e_N)
\nonumber
\\
&\leq& \sup_{x\in\bar{\Omega}}|u+u_N|\cdot\|e_N+\eta_N\|\|\nabla
e_N\| \nonumber
\\
&\leq&2(c_{10}+c_{11})(\|e_N\|\|\nabla e_N\|+\|\eta_N\|\|\nabla
e_N\|) \nonumber
\\
&\leq&-\frac{\gamma}4(e_N,\Delta e_N)+c_{13}(\|e_N\|^2+\|\eta_N\|^2)
\nonumber \\
&\leq&\frac{\gamma}4\|\Delta
e_N\|^2+(c_{13}+\frac{\gamma}{16})\|e_N\|^2+c_{13}\|\eta_N\|^2,
\nonumber
\end{eqnarray}
where $c_{13}=\frac4{\gamma}(c_{10}+c_{11})^2$. From Theorem
\ref{thm1.1} and (B2), we have
$$
\|\eta_N\|\leq cN^{-2}\|\Delta u\|\leq c_{14}N^{-2}.
$$
Summing up, we immediately obtain
\begin{eqnarray}
\frac d{dt}\|e_N\|^2+\gamma\|\Delta
e_N\|^2\leq(2c_{12}+2c_{13}+\frac{\gamma}8)\|e_N\|^2+2(c_{12}+c_{13})c_{14}N^{-4}.\nonumber
\end{eqnarray}
Therefore, by Gronwall's inequality, we deduce that
\begin{eqnarray}
\label{2-28} \|e_N\|\leq c(\|e_N(0)+N^{-2}).
\end{eqnarray}

Thus, we obtain the following theorem:
\begin{theorem}
\label{thm2.5} Suppose that $u_0\in H^2_{per}(\Omega)$, $u(x,t)$ is
the solution of problem (\ref{1-1})-(\ref{1-3}) and $u_N(x,t)$ is
the solution of semi-discrete approximation (\ref{2-2}). Then, there
exists a constant $c$, independent of $N$, such that
$$
\|u(x,t)-u_N(x,t)\|\leq c(N^{-2}+\|u_0-u_N(0)\|).
$$
\end{theorem}

\section{Fully discrete scheme}
In this section, we set up a full-discretization scheme for problem
(\ref{1-1})-(\ref{1-3}) and consider the fully discrete scheme which
implies the pointwise bounded of the solution.

Let $\Delta t=T/M$, for a positive integer $M$, $\bar{\partial}_tu^k=\frac{u^k-u^{k-1}}{\Delta t}$. Note that $\varphi(s)=\gamma_2s^2+\gamma_1s-1$ and $\psi(s)=s^2$. The
full-discretization spectral method for problem
(\ref{1-1})-(\ref{1-3}) is read as: find $u_N^j\in
S_N~(j=0,1,2,\cdots,k)$ such that for any $v_N\in S_N$, there hold
\begin{equation}\begin{aligned}
\label{3-1} &\left(\frac{u_N^{k}-u_N^{k-1}}{\Delta
t},v_N\right)+\gamma(\Delta u_{N}^{k},\Delta
v_{N})
+\left(\varphi'(u_N^{k-1})\nabla u_N^k,\nabla v_N\right)\\
&-\frac23\left(u_N^{k-1}\nabla\cdot u_N^k,v_N\right)+\frac23\left(u_N^{k-1}u_N^k,\nabla\cdot v_N\right)=0.\end{aligned}\end{equation}
for all $T>0$ and $t\in[ 0,T]$ with $u_N(0)=P_Nu_0$.

The solution $u_N^k$ has the following property:
\begin{lemma}
\label{lem3.1} Suppose that $u_0\in H^2(\Omega)\bigcap H_0^1(\Omega)$ and $u_N^k$
is a solution of problem (\ref{3-1}), then there exists positive
constants $c_{15}, c_{16}, c_{17}, c_{18}$ depend only on $\gamma$,
$\gamma_1$, $\gamma_2$ and $u_0$, such that
$$
\|u_N^k\|\leq c_{15}, ~~\|\nabla u_{N}^k\|\leq c_{16},~~\|\Delta
u_{N}^k\|\leq c_{17},~~ \|u_{N}^k\|_{\infty}\leq c_{18}.
$$
\end{lemma}
\begin{proof}
Let $v_N=u_N^k$ in (\ref{3-1}), we derive that
\begin{equation}\begin{aligned}&
\frac12\bar{\partial}_t\|u_N^k\|^2+\frac{\tau}2\|\bar{\partial}_tu_N^k\|^2+\gamma\|\Delta u_{N}^k\|^2+(\varphi'(u_N^{k-1})\nabla u_{N}^k,\nabla u_N^k)\\&-\frac23\left(u_N^{k-1}\nabla\cdot u_N^k,u^k_N\right)+\frac23\left(u_N^{k-1}u_N^k,\nabla\cdot u^k_N\right)=0.
\end{aligned}\nonumber\end{equation}
Note that $$
\varphi'(u_N^{k-1})=3\gamma_2(u^{k-1}_N)^2+2\gamma_1u_N^{k-1}-u_N^{k-1}\geq-2c_0=-\frac{\gamma_1^2}{3\gamma_2}-1.
$$Thus
$$
(\varphi'(u_N^{k-1})\nabla u_{N}^k,\nabla u_N^k)\geq-c_0\|\nabla u_N^k\|^2.
$$
On the other hand, we have
$$
-\frac23\left(u_N^{k-1}\nabla\cdot u_N^k,u^k_N\right)+\frac23\left(u_N^{k-1}u_N^k,\nabla\cdot u^k_N\right)=0.
$$
Therefore
\begin{equation}
\bar{\partial}_t\|u_N^k\|^2+\tau\|\bar{\partial}_tu_N^k\|^2+2\gamma\|\Delta u_{N}^k\|^2\leq c_0\|\nabla u_N^k\|^2.\label{zxp-1}
\end{equation}
Note that
\begin{equation}\label{spring2-1}
c_0\|\nabla u_N^k\|^2\leq\frac{c_0^2}{4\gamma}\|u_N^k\|^2+\gamma\|\Delta u_N^k\|^2.
\end{equation}
Summing up, we get
\begin{equation}
\label{spring-1}
\frac{\|u_N^k\|^2-\|u_N^{k-1}\|^2}{\Delta t}+\gamma\|\Delta u_N^k\|^2\leq\frac{c_0^2}{4\gamma}\|u_N^k\|^2,
\end{equation}
that is
\begin{equation}
\label{spring-2}
\|u_N^k\|^2\leq\frac{4\gamma}{4\gamma-c_0^2\Delta t}\|u_N^{k-1}\|^2\leq\left(\frac{4\gamma}{4\gamma-c_0^2\Delta t}\right)^k\|u_N^0\|^2=c_{15}.
\end{equation}

Let $\varphi=\Delta u_{N}^k$ in (\ref{3-1}), we derive that
\begin{equation}\begin{aligned}&
\frac12\bar{\partial}_t\|\nabla u_{N}^k\|^2+\frac{\tau}2\|\bar{\partial}_t\nabla u_{N}^k\|^2+\gamma\|\nabla\Delta u_{N}^k\|^2
\\=&(\varphi'(u_{N}^{k-1})\nabla u^k_{N},\nabla\Delta u^k_{N})-\frac43\left(u_N^{k-1}\nabla\cdot u_N^k,\Delta u^k_N\right)-\frac23\left(\nabla u_N^{k-1}u_N^k,\Delta u^k_N\right).\end{aligned}\label{2-2a}
\end{equation}
By Young's inequality, Sobolev's interpolation inequality and (\ref{spring-2}), we get
\begin{equation}
\begin{aligned}&(\varphi'(u_{N}^{k-1})\nabla u^k_{N},\nabla\Delta u^k_{N})\\
\leq &c(\|u_N^{k-1}\|_{L^4}^2+1)\|\nabla u_N^k\|_{\infty}\|\nabla\Delta u_N^k\|
\\
\leq&c(\|\nabla u_N^{k-1}\|\|u_N^{k-1}\|+1)(\|\nabla\Delta u_N^k\|^{\frac23}\|u_N^k\|^{\frac13}+1)\|\nabla\Delta u_N^k\|
\\
\leq&\frac{\gamma}4\|\nabla\Delta u^k_{N}\|^2+c(\|\nabla u_{N}^{k-1}\|^2+1).\nonumber
\end{aligned}\end{equation}
Using Young's inequality and Sobolev's interpolation inequality again, we get
\begin{equation}
\begin{aligned}&
-\frac43\left(u_N^{k-1}\nabla\cdot u_N^k,\Delta u^k_N\right)\\
\leq&c\|u_N^{k-1}\|\|\nabla u_N^k\|_{L^4}\|\Delta u_N^k\|_{L^4}
\\\leq& c\| u_N^{k-1}\|\|\nabla\Delta u_N^k\|^{\frac12}\|u_N^k\|^{\frac12}\|\nabla\Delta u_N^k\|^{\frac56}\|u_N^k\|^{\frac16}
\\
\leq&\frac{\gamma}8\|\nabla\Delta u_{N}^k\|^2+c,
\end{aligned}\nonumber
\end{equation}
and
\begin{equation}
\begin{aligned}
-\frac23\left(\nabla u_N^{k-1}u_N^k,\Delta u_N^k\right)
\leq&\frac23\|\nabla u_N^{k-1}\|\|u_N^k\|_{\infty}\|\Delta u_N^k\|
\\
\leq&\frac23\|\nabla u_N^{k-1}\|\|\nabla\Delta u_N^k\|\|u_N^k\|
\\
\leq&\frac{\gamma}8\|\nabla\Delta u_N^k\|^2+c\|u_N^k\|^2\|\nabla u_N^{k-1}\|^2
\\
\leq&\frac{\gamma}8\|\nabla\Delta u_N^k\|^2+c\|u_N^k\|^2.
\end{aligned}\nonumber
\end{equation}
Hence, (\ref{2-2a}) can be rewritten as
\begin{equation}
\begin{aligned}
\frac{\|\nabla u_N^k\|^2-\|\nabla u_N^{k-1}\|^2}{\Delta t}+\gamma\|\nabla\Delta u_N^k\|^2\leq c(\|\nabla u_N^{k-1}\|^2+1).
\label{spring-3}\end{aligned}
\end{equation}
Using discrete Gronwall's inequality, we deduce that
\begin{equation}
\label{sprin-4}
\|\nabla u_N^k\|^2\leq (\|\nabla u_0\|^2+ct_n)e^{ct_n}\leq c_{16}.
\end{equation}
By Sobolev's embedding theorem, we have
\begin{equation}
\label{spring-5}
\|u_N^k\|_{L^p}\leq c_{19},\quad 1<p<\infty.
\end{equation}

Let $\varphi=\Delta^2u_{N}^k$ in (\ref{3-1}), we have
\begin{equation}\begin{aligned}&
\frac12\bar{\partial}_t\|\Delta u_{N}^k\|^2+\frac{\tau}2\|\bar{\partial}_t\Delta u_{N}\|^2+\gamma\|\Delta^2u_{N}^k\|^2
\\=&\left(\nabla\cdot[\varphi'(u_{N}^{k-1})\nabla\cdot u_{N}^k],\Delta^2u_{N}^k\right)+\frac23\left(u_N^{k-1}\nabla\cdot u_N^k,\Delta^2u_{N}^k\right)\\&+\frac23\left(\nabla\cdot(u_N^{k-1}u_N^k),\Delta^2u_{N}^k\right).
\nonumber\end{aligned}\end{equation}
Based on the above results and Sobolev's interpolation inequality, we deduce that
\begin{equation}
\begin{aligned}&
\left(\nabla\cdot[\varphi'(u_{N}^{k-1})\nabla\cdot u_{N}^k],\Delta^2u_{N}^k\right)\\=&(\varphi'(u_N^{k-1})\Delta u_{N}^k,\Delta^2u_{N}^k)+(\varphi{''}(u_N^{k-1})|\nabla u_{N}^k)|^2,\Delta^2u_{N}^k)
\\
\leq&\|\varphi'(u_N^{k-1})\|_{L^4}\|\Delta u_{N}^k\|_{L^4}\|\Delta^2u_{N}^k\|+\|\varphi{''}(u_N^{k-1})\|_{L^6}\|\nabla u_{N}^k\|_{L^6}^2\|\Delta^2u_{N}^k\|
\\
\leq&c\|\Delta u_{N}^k\|_{L^4}\|\Delta^2u_{N}^k\|+c\|\nabla u_{N}^k\|_{L^6}^2\|\Delta^2u_{N}^k\|
\\
\leq&c\|\Delta^2u_N^k\|^{\frac32}\|\nabla u_N^k\|^{\frac12}+c\|\Delta^2u_N^k\|^{\frac{13}9}\|\nabla u_N^k\|^{\frac{14}9}\\
\leq&\frac{\gamma}4\|\Delta ^2u_{N}^k\|^2+c(c_{15},c_{16},c_{19}).
\nonumber\end{aligned}\end{equation}
We also have
\begin{equation}\begin{aligned}
\frac23\left(u_N^{k-1}\nabla\cdot u_N^k,\Delta^2u_{N}^k\right)\leq&\frac23\|u_N^{k-1}\|_{L^4}\|\nabla u_N^k\|_{L^4}\|\Delta^2u_N^k\|
\\
\leq&c\|\nabla u_N^k\|_{L^4}\|\Delta^2u_N^k\|\leq c\|\Delta^2u_N^k\|^{\frac32}\|\nabla u_N^k\|^{\frac12}
\\
\leq&\frac{\gamma}8\|\Delta ^2u_{N}^k\|^2+c(c_{15},c_{16},c_{19}),
\nonumber\end{aligned}\end{equation}
and
\begin{equation}\begin{aligned}&\frac23\left(\nabla\cdot(u_N^{k-1}u_N^k),\Delta^2u_{N}^k\right)\\=&\frac23[(u_N^k\nabla\cdot u_N^{k-1},\Delta^2u_{N}^k)+
(u_N^{k-1}\nabla\cdot u_N^{k},\Delta^2u_{N}^k)]
\\
\leq&\frac23\|u_N^k\|_{\infty}\|\nabla u_N^{k-1}\|\|\Delta^2u_N^k\|+\frac23\|u_N^{k-1}\|_{L^4}\|\nabla u_N^k\|_{L^4}\|\Delta^2u_N^k\|
\\
\leq&c\|\nabla u_N^{k-1}\|\|\Delta^2u_N^k\|^{\frac54}\|u_N^k\|^{\frac34}+c\|\Delta^2u_N^k\|^{\frac32}\|\nabla u_N^k\|^{\frac12}
\\
\leq&\frac{\gamma}8\|\Delta ^2u_{N}^k\|^2+c(c_{15},c_{16},c_{19}).
\nonumber\end{aligned}\end{equation}
Summing up, we derive that
\begin{equation}
\label{2-4}
\frac{\|\Delta u_N^k\|^2-\|\Delta u_N^{k-1}\|^2}{\Delta t}+\gamma\|\Delta^2u_N^k\|^2\leq c.
\end{equation}
Therefore
\begin{equation}
\label{spring-6}
\|\Delta u_N^k\|^2\leq c\Delta t+\|\Delta u_N^{k-1}\|^2\leq cT+\|\Delta u_0\|^2=c_{17}.
\end{equation}
By Sobolev's embedding theorem, we have
\begin{equation}
\label{spring-5}
\|u_N^k\|_{\infty}\leq c_{18}.
\end{equation}

Then, the proof is complete.

\end{proof}
In the following, we analyze the error estimates between the
numerical solution $u_N^k$ and the exact solution $u(t_k)$.

We introduce a linear problem as follows: $\forall v\in S_N$,
\begin{equation} \label{spring-6}
\quad\left\{ \begin{aligned}
         &\left(\frac{w_N^k-w_N^{k-1}}{\Delta t}+\gamma\Delta^2w_N^k+\gamma w^k_N-\Delta\varphi(u^k)-\nabla\cdot\psi(u^k),v_N\right)=(\gamma u_N^k,v_N),
 \\
                  &w_N^0=P_Nu_0.
                          \end{aligned} \right.
                          \end{equation}
First of all, we study the error estimates between $u(t_k)$ and $w_N^k$.
Set $u^k=u(t_k)$, $\eta^k=u^k-P_Nu^k$ and $\theta^k=P_Nu^k-w_N^k.$ Then, we have
$$u^k-w_N^k=u^k-P_Nu^k+P_Nu^k-w_N^k=\eta^k+\theta^k.
$$
\begin{lemma}
\label{lem.1}
Suppose that $u^k=u(t_k)$ is the solution of problem (\ref{1-1})-(\ref{1-3}) and $w_N^k$ is the solution of problem (\ref{spring-6}). Suppose further that $u_{tt}\in L^2(0,T;L^2(\Omega))$. Then, we have
$$
\|u^k-w_N^k\|^2\leq \|u^0-w_N^0\|^2+c(\Delta t)^2.
$$
\end{lemma}
\begin{proof}
Note that $u^k-w_N^k=\eta^k+\theta^k$. $\theta^k$ satisfies
\begin{equation}
\label{spring-8}
\left(\frac{\theta^k-\theta^{k-1}}{\Delta t}+\gamma\Delta^2\theta^k+\gamma\theta^k-\left(\frac{u^k-u^{k-1}}{\Delta t}-u_t\right),v_N\right)=0.
\end{equation}
Set $v_N=\theta^k$, we get
\begin{equation}
\frac12\frac{\|\theta^k\|^2-\|\theta^{k-1}\|^2}{\Delta t}+\gamma\|\Delta\theta^k\|^2+\gamma\|\theta^k\|^2=\left(\frac{u_N^k-u_N^{k-1}}{\Delta t}-u_t^k,\theta^k\right).
\nonumber\end{equation}
Note that
\begin{equation}
\begin{aligned}&
\left(\frac{u_N^k-u_N^{k-1}}{\Delta t}-u_t^k,\theta^k\right)\\\leq&\|\theta^k\|\left\|\frac{u_N^k-u_N^{k-1}}{\Delta t}-u_t^k\right\|
\\
\leq&\gamma\|\theta^k\|^2+\frac1{(\Delta t)^2}\left\|\int_{t_{k-1}}^{t_k}(s-t_{k-1})u_{tt}(\xi^k)ds\right\|^2
\\
\leq&\gamma\|\theta^k\|^2+\frac1{(\Delta t)^2}\int_{t_{k-1}}^{t_k}(s-t_{k-1})^2ds\int_{t_{k-1}}^{t_k}\|u_{tt}(\xi^k)\|^2ds
\\
\leq&\gamma\|\theta^k\|^2+\Delta t\int_{t_{k-1}}^{t_k}\|u_{tt}(\xi^k)\|^2ds,\nonumber
\end{aligned}\end{equation}
where $t_{k-1}<\xi^k<t_k$. Summing up, we derive that
\begin{equation}
\frac{\|\theta^k\|^2-\|\theta^{k-1}\|^2}{\Delta t}\leq\Delta t\int_{t_{k-1}}^{t_k}\|u_{tt}(\xi^k)\|^2ds,
\nonumber\end{equation}
that is
\begin{equation}
\|\theta^k\|^2\leq \|\theta^{k-1}\|^2+(\Delta t)^2\int_{t_{k-1}}^{t_k}\|u_{tt}(\xi^k)\|^2ds.
\nonumber\end{equation}
Therefore
\begin{equation}
\|\theta^k\|^2\leq \|\theta^{0}\|^2+(\Delta t)^2\sum_{i=1}^k\int_{t_{k-1}}^{t_k}\|u_{tt}(\xi^i)\|^2ds,\nonumber
\end{equation}
where $i=1,2,\cdots,k$ and $t_{i-1}<\xi^i<t_i$. Set $\|u_{tt}(\xi)\|=\max\{\|u_{tt}(\xi^i)\|\}$, we have
\begin{equation}
\begin{aligned}
\|\theta^k\|^2\leq& \|\theta^{0}\|^2+(\Delta t)^2\sum_{i=1}^k\int_{t_{k-1}}^{t_k}\|u_{tt}(\xi)\|^2ds
\\
=&\|\theta^{0}\|^2+(\Delta t)^2\int_0^T\|u_{tt}(\xi)\|^2ds
\\
\leq&\|\theta^{0}\|^2+c(\Delta t)^2.\nonumber\end{aligned}
\end{equation}
Hence, the proof is complete.

\end{proof}

Secondly, we study the error estimate between $w_N^k$ and $u_N^k$. Set $w_N^k-u_N^k=e_N^k$. We have the following lemma:

\begin{lemma}
\label{lem.2}
Suppose that $w_N^k$ is the solution of problem (\ref{spring-6}) and $u_N^k$ is the solution of the full-discrete scheme (\ref{3-1}). Suppose further that $u_t\in L^2(0,T;L^2(\Omega))$. Then, we have
$$
\|w_N^k-u_N^k\|^2\leq c((\Delta t)^2+N^{-4}).
$$
\end{lemma}
\begin{proof}
Combining (\ref{3-1}) and (\ref{spring-6}) together gives
\begin{equation}
\begin{aligned}
&\left(\frac{e^k-e^{k-1}}{\Delta t}+\gamma\Delta^2e^k-\gamma\theta^k,v_N\right)+(\nabla\varphi(u^k)-\varphi'(u_N^{k-1})\nabla u_N^k,\nabla v_N)
\\
=&(\nabla\cdot(u^k)^2,v_N)-\frac23(u_N^{k-1}\nabla\cdot u_N^k,v_N)+\frac23(u_N^{k-1}u_N^k,\nabla\cdot v_N).
\label{spring-9}\end{aligned}
\end{equation}
Set $v_N=e^k$ in (\ref{spring-9}), we get
\begin{equation}\begin{aligned}
\label{spring-10}&
\frac12\frac{\|e^k\|^2-\|e^{k-1}\|^2}{\Delta t}+\gamma\|\Delta e^k\|^2
\\
=&\gamma(\theta^k,e^k)-(\nabla\varphi(u^k)-\varphi'(u_N^{k-1})\nabla u_N^k,\nabla e^k)
\\
&+(\nabla\cdot(u^k)^2,e^k)-\frac23(u_N^{k-1}\nabla\cdot u_N^k,e^k)+\frac23(u_N^{k-1}u_N^k,\nabla\cdot e^k).
\nonumber\end{aligned}
\end{equation}
Note that
\begin{equation}
\begin{aligned}&
-(\nabla \varphi(u^k)-\varphi'(u_N^{k-1})\nabla u_{N}^k,\nabla e^k)
=-(\varphi'(u^k)\nabla u^k-\varphi'(u_N^{k-1})\nabla u_{N}^k,\nabla e^k)
\\
=&-(\varphi'(u^k)\nabla u^k-\varphi'(u^k)\nabla u_{N}^k+\varphi'(u^k)\nabla u_{N}^k-\varphi'(u^{k-1})\nabla u_{N}^k
\\
&+\varphi'(u^{k-1})\nabla u_{N}^k-\varphi'(u_N^{k-1})\nabla u_{N}^k,\nabla e^k)
\\
=&-(\varphi'(u^k)(\nabla u^k-\nabla u_{N}^k),\nabla e^k)-([\varphi'(u^k)-\varphi'(u^{k-1})]\nabla u_{N}^k,\nabla e^k)
\\&-([\varphi'(u^{k-1})-\varphi'(u_N^{k-1})]\nabla u_{N}^k,\nabla e^k)
\\
\triangleq&I_1+I_2+I_3.
\end{aligned}
\nonumber\end{equation}
\begin{equation}
\begin{aligned}
I_1=&(\varphi'(u^k)\Delta e^k+\varphi{''}(u^k)\nabla u^k\nabla e^k,u^k-u_N^k)
\\
\leq&(\|\varphi'(u^k)\|_{\infty}\|\Delta e^k\|+\|\varphi{''}(u^k)\|_{\infty}\|\nabla u^k\|_{L^4}\|\nabla e^k\|_{L^4})\|u^k-u_N^k\|
\\
\leq&(\|\varphi'(u^k)\|_{\infty}\|\Delta e^k\|+C\|\varphi{''}(u^k)\|_{\infty}\|\nabla u^k\|_{L^4}\|\Delta e^k\|)\|u^k-u_N^k\|\\
\leq& c\|u^k-u_N^k\|(\|e^k\|+\|\Delta e^k\|)
\\
\leq&c(\|e^{k}\|+\|\eta^{k}\|+\|\theta^{k}\|)(\|e^k\|+\|\Delta e^k\|)
\\
\leq& \varepsilon\|\Delta e^k\|^2+c(\|\eta^k\|+\|\theta^k\|+\|e^k\|).
\nonumber
\end{aligned}
\end{equation}
\begin{equation}
\begin{aligned}I_2=&-(\varphi{''}(\phi_1u^k+(1-\phi_1)u^{k-1})\nabla u_{N}^k(u^k-u^{k-1}),\nabla e^k)
\\
\leq&\|\varphi{''}(\lambda_1u^k+(1-\lambda_1)u^{k-1})\|_{\infty}\|\nabla u_{N}^k\|_{L^4}\|u^k-u^{k-1}\|\|\nabla e^k\|_{L^4}
\\\leq&c\|\varphi{''}(\lambda_1u^k+(1-\lambda_1)u^{k-1})\|_{\infty}\|\nabla u_{N}^k\|_{L^4}\|u^k-u^{k-1}\|(\|e^k\|+\|\Delta e^k\|)
\\\leq& c(\|e^k\|+\|\Delta e^k\|)\|u^k-u^{k-1}\|
\\\leq&\varepsilon\|\Delta e^k\|^2+c\Delta t\int_{t_{k-1}}^{t_k}\|u_t\|^2ds,
\nonumber
\end{aligned}
\end{equation}
\begin{equation}
\begin{aligned}
I_3=&-(\varphi{''}(\lambda_2u^{k-1}+(1-\lambda_2)u_N^{k-1})\nabla u_{N}^k(u^{k-1}-u_N^{k-1}),\nabla e^k)
\\
\leq&\|\varphi{''}(\lambda_2u^{k-1}+(1-\lambda_2)u_N^{k-1})\|_{\infty}\|\nabla u_{N}^k\|_{L^4}\|u^{k-1}-u_N^{k-1}\|\|\nabla e^k\|_{L^4}
\\\leq& c(\|e^k\|+\|\Delta e^k\|)\|u^{k-1}-u_N^{k-1}\|
\\
\leq&\varepsilon\|\Delta e^k\|^2+c(\|\eta^{k-1}\|^2+\|\theta^{k-1}\|^2+\|e^{k-1}\|^2),
\nonumber
\end{aligned}
\end{equation}
where $\lambda_1,\lambda_2\in(0,1)$. Hence
\begin{equation}\begin{aligned}&
(\nabla \varphi(u^k)-\varphi'(u_N^{k-1})\nabla u_{N}^k,\nabla e^k)\\\leq& 3\varepsilon\|\Delta e^k\|^2
+c(\|\eta^{k-1}\|^2+\|\theta^{k-1}\|^2+\|e^{k-1}\|^2+\|\eta^{k}\|^2
\\&+\|\theta^{k}\|^2+\|e^{k}\|^2+\Delta t\int_{t_{k-1}}^{t_k}\|u_t\|^2ds
).
\nonumber
\end{aligned}
\end{equation}
We also have
\begin{equation}
\begin{aligned}
&(\nabla\cdot(u^k)^2,e^k)-\frac23(u_N^{k-1}\nabla\cdot u_N^k,e^k)+\frac23(u_N^{k-1}u_N^k,\nabla\cdot e^k)
\\
=&\frac23(u^k\nabla\cdot u^k-u^{k-1}_N\nabla\cdot u_N^k,e^k)+\frac43(u^k\nabla\cdot u^k,e^k)-\frac23(u_N^{k-1}u_N^k,\nabla\cdot e^k)
\\
=&\frac23(u^k\nabla u^k-u^{k-1}_N\nabla\cdot u_N^k,e^k)-\frac23([u^k]^2-u_N^{k-1}u_N^k,\nabla\cdot e^k)
\\
=&\frac23[(u^k-u_N^k,e^k\nabla\cdot u^k)+(u_N^k-u_N^{k-1},e^k\nabla\cdot u^k)-(u^k-u^{k-1},u^k\nabla \cdot e^k)
\\& -(u^k-u^k_N,u^{k-1}\nabla\cdot e^k)-(u^{k-1}-u_N^{k-1},u_N^k\nabla e^k]
\\
\leq&\varepsilon(\|e^k\|^2+\|\Delta e^k\|^2)
+c(\|\eta^{k-1}\|^2+\|\theta^{k-1}\|^2+\|e^{k-1}\|^2+\|\eta^{k}\|^2\\&+\|\theta^{k}\|^2+\|e^{k}\|^2+\Delta t\int_{t_{k-1}}^{t_k}\|u_t\|^2ds),
\nonumber
\end{aligned}
\end{equation}and
$$
\gamma(\theta^k,e^k)\leq\varepsilon\|e^k\|^2+c\|\theta^k\|^2.
$$
Summing up, we immediately obtain
\begin{equation}
\begin{aligned}&
\frac{\|e^k\|^2-\|e^{k-1}\|^2}{\Delta t}+2(\gamma-5\varepsilon)\|\Delta e^k\|^2
\\
\leq&c(\|\eta^{k-1}\|^2+\|\eta^k\|^2+\|\theta^{k-1}\|^2+\|\theta^k\|^2+\|e^{k-1}\|^2+\|e^k\|^2+\Delta t\int_{t_{k-1}}^{t_k}\|u_t\|^2ds),
\nonumber
\end{aligned}
\end{equation}where $\varepsilon$ is small enough, which satisfies $\gamma-5\varepsilon>0$. Thus
\begin{equation}\begin{aligned}&
(1-c\Delta t)\|e^k\|^2+2(\gamma-5\varepsilon)\Delta t\|\Delta e^k\|^2
\\
\leq&(1+c\Delta t)\|e^{k-1}\|^2\\&+c\Delta t(\|\eta^{k-1}\|^2+\|\theta^{k-1}\|^2\|\eta^{k}\|^2+\|\theta^{k}\|^2+\Delta t\int_{t_{k-1}}^{t_k}\|u_t\|^2ds),
\nonumber
\end{aligned}
\end{equation}that is
\begin{equation}\begin{aligned}
\|e^k\|^2\leq&\frac{1+c\Delta t}{1-c\Delta}\|e^{k-1}\|^2
\\&+\frac{c\Delta t}{1-c\Delta t}(\|\eta^{k-1}\|^2+\|\theta^{k-1}\|^2+\|\eta^{k}\|^2+\|\theta^{k}\|^2+\Delta t\int_{t_{k-1}}^{t_k}\|u_t\|^2ds).\end{aligned}
\nonumber\end{equation}
Applying the discrete Gronwall's inequality with sufficient small $\Delta t$ such that $1-c\Delta t >0$, we get
$$
\|e^n\|\leq c(\Delta t+N^{-2}).
$$
Hence, we complete the proof.

\end{proof}

Furthermore, we have the following theorem:
\begin{theorem}
\label{thmm}
Suppose that $u_0\in H^2(\Omega)\bigcap H_0^1(\Omega)$, $u(x,t)$ is the solution of problem (\ref{1-1})-(\ref{1-3}) satisfying
$$
u_t\in L^2(0,T;L^2(\Omega)),\quad u_{tt}\in L^2(0,T;L^2(\Omega)).
$$
Suppose further that $u_N^k\in S_N~(k=0,1,2,\cdots)$ is the solution for problem (\ref{3-1}) and the initial value $u_N^0$ satisfies
$$
\|u_N^0-P_Nu_0\|\leq cN^{-2}\|\Delta u\|.
$$
Then, there exists a positive constant $c$ depends on $\gamma$, $\gamma_1$, $\gamma_2$, $T$ and $u_0$, independent of $N$ such that
$$
\|u(x,t_k)-u_N^k\|\leq c(\Delta t+N^{-2}),\quad j=0,1,2,\cdots,N.
$$
\end{theorem}

\section*{Acknowledgement}

This paper was supported by the National Natural Science Foundation of China (grant
No. 11401258)   and China Postdoctoral Science Foundation (grant No. 2015M581689
). 






\end{document}